\documentclass{paper}

\usepackage{import}
\usepackage{csquotes}

\usepackage[all]{xy}
\usepackage{euscript}
\usepackage{blkarray}
\usepackage{nicematrix}

\theoremstyle{definition}
\newtheorem{convention}[theorem]{Convention}

\theoremstyle{theorem}
\newtheorem{question}[theorem]{Question}

\def\PP{\mathbb{P}}
\def\OO{\mathcal{O}}
\def\FF{\mathbb{F}}

\def\a{\alpha}
\def\b{\beta}
\def\g{\gamma}
\def\del{\partial}
\def\cech{\check{C}}

\def\ub{\underbrace}
\def\on{\operatorname}

\def\Sym{\operatorname{Sym}}

\def\E{\mathcal{E}}

\def\B{\mathbf{B}}
\def\RR{\mathbf{R}}
\def\cR{\mathcal{R}}

\def\F{\mathcal{F}}
\def\G{\mathcal{G}}
\def\C{\mathcal{C}}

\def\Db{\operatorname{D}^{\operatorname{b}}} 
\def\coh{\operatorname{coh}}
\def\cL{\mathcal{L}}

\def\from{\leftarrow}

\def\th{\operatorname{th}}

\def\cT{\mathcal{T}}
\def\cK{\mathcal{K}}
\def\cB{\mathcal{B}}

\def\cW{\mathcal{W}}
\def\cV{\mathcal{V}}
\def\cC{\mathcal{C}}

\def\MR#1{}
\newcommand{\cP}{\mathcal{P}}
\newcommand{\cQ}{\mathcal{Q}}

\newcommand{\red}[1]{{\color{red}#1}}

\title{A short resolution of the diagonal for smooth projective toric varieties of Picard rank 2}

\author{Michael K. Brown}
\address{Department of Mathematics, Auburn University, Auburn, Alabama, 36849}
\email{\href{mailto:mkb0096@auburn.edu}{mkb0096@auburn.edu}}
\urladdr{\url{http://webhome.auburn.edu/~mkb0096/}}

\author{Mahrud Sayrafi}
\address{School of Mathematics, University of Minnesota, Minneapolis, Minnesota, 55455}
\email{\href{mailto:mahrud@umn.edu}{mahrud@umn.edu}}
\urladdr{\url{https://math.umn.edu/~mahrud/}}

\makeatletter
\@namedef{subjclassname@2020}{%
  \textup{2020} Mathematics Subject Classification}
\makeatother

\begin{document}

\begin{abstract}
Given a smooth projective toric variety $X$ of Picard rank 2, we resolve the diagonal sheaf on $X \times X$ by a linear complex of length $\dim{X}$ consisting of finite direct sums of line bundles. As applications, we prove a new case of a conjecture of Berkesch--Erman--Smith that predicts a version of Hilbert's Syzygy Theorem for virtual resolutions, and we obtain a Horrocks-type splitting criterion for vector bundles over smooth projective toric varieties of Picard rank~2, extending a result of Eisenbud--Erman--Schreyer. We also apply our results to give a new proof, in the case of smooth projective toric varieties of Picard rank~2, of a conjecture of Orlov concerning the Rouquier dimension of derived categories.
\end{abstract}
\maketitle


\section{Introduction}

Beilinson's resolution of the diagonal over a projective space is a powerful tool in algebraic geometry \cite{Bei78}. For instance, this resolution may be used to show that the bounded derived category $\Db(\PP^n)$ is generated by the line bundles $\OO, \OO(1), \dots, \OO(n)$. Additionally, taking a Fourier--Mukai transform with kernel given by Beilinson's resolution yields a representation of any object in $\Db(\PP^n)$ as a complex of vector bundles, called a \emph{Beilinson monad}, which has been used to great effect in computational algebraic geometry, e.g.~\cite{ES03,ES09}.

In this paper, we aim to construct a Beilinson-type resolution of the diagonal over a smooth projective toric variety $X$ of Picard rank 2. More specifically, with a view toward proving a new case of a conjecture of Berkesch--Erman--Smith (Conjecture~\ref{conj:BES} below), we construct such a resolution of length $\dim{X}$---the shortest possible length---whose terms are finite direct sums of line bundles. While the existence of a full strong exceptional collection of line bundles \cite{CM04,BH09} implies that $X$ admits a resolution of the diagonal via a tilting bundle construction \cite[Prop.~3.1]{King97}, it follows from a result of Ballard--Favero~\cite[Prop.~3.33]{BF12} that this resolution may have length greater than $\dim{X}$.
Our main result is as follows:

\begin{theorem}\label{thm:main}
Let $X$ be the projective bundle $\PP(\OO \oplus \OO(a_1) \oplus \cdots \oplus \OO(a_s))$ over $\PP^r$, where $1\leq r,s$ and $0 \le a_1 \le \cdots \le a_s$. Denote by $\FF_{a_s}$ the Hirzebruch surface of type $a_s$, and equip $\Pic(\FF_{a_s}) \cong \ZZ^2$ with the basis described in Convention~\ref{conv:basis} below.  There is a complex $R$ of finitely generated graded free modules over the Cox ring of $X \times X$ such that:
\begin{enumerate}
\item $R$ is exact in positive degrees.
\item $R$ is linear, in the sense that there exists a basis of $R$ with respect to which the differentials of $R$ are matrices whose entries are $\kk$-linear combinations of the variables.
\item We have $\on{rank} R_n = \binom{r+s}{n}\dim_\kk H^0(\FF_{a_s}, \OO(r,s))$. In particular, $R$ has length $\dim X = r+s$, and the equality $\on{rank} R_n = \on{rank} R_{r + s - n}$ holds.
\item The sheafification $\cR$ of $R$ is a resolution of the diagonal sheaf $\OO_\Delta$ on $X \times X$.
\end{enumerate}
\end{theorem}

We note that, by a result of Kleinschmidt in \cite{Kleinschmidt88}, every smooth projective toric variety of Picard rank 2 arises as a projective bundle as in the hypothesis of Theorem~\ref{thm:main}. We construct the resolution $\cR$ in Theorem~\ref{thm:main} using a variant of Weyman's ``geometric technique" for building free resolutions, described in \cite[\S5]{Weyman2003}. In a bit more detail: let $x_i$ and $x'_i$ refer to the variables corresponding to the first and second copy of $X$, respectively, in the Cox ring $S$ of $X \times X$. A first, naive, idea is that the diagonal sheaf $\OO_\Delta$ ought to be defined by the relations $x_i - x'_i$ in $S$. The problem is that these relations are not homogeneous with respect to the $\ZZ^4$-grading on $S$. To fix this, we homogenize the relations $x_i - x'_i$ in the Cox ring of a certain toric fiber bundle $E$ over $X \times X$ with fiber given by $\FF_{a_s}$. Our resolution $\cR$ is obtained by taking the Koszul complex on these homogenized relations over $E$, twisting it by a certain line bundle, and pushing it forward to $X \times X$. Choosing the toric fiber bundle $E$ is delicate; not only do the degrees of the variables in the Cox ring of $E$ need to be suitable for homogenizing the relations $x_i - x_i'$, but the terms of the Koszul complex on these homogenized relations must enjoy appropriate cohomological vanishing properties in order to conclude that $\cR$ is a resolution of the required form. See \S\ref{sec:bundle} for details.

The simplest case of Theorem~\ref{thm:main} is the Hirzebruch surface $\FF_{a} = \PP(\OO \oplus \OO(a))$, where $r = s = 1$ and $a = a_1$. As detailed in Example~\ref{ex:hirzebruch}, the construction above yields a resolution of the diagonal for $\FF_{a}$ whose terms $\cR_0$, $\cR_1$, and $\cR_2$ are sums of $a + 4$, $2a + 8$, and $a + 4$ line bundles, respectively (cf.~\cite[\S1]{Buchdahl87}).

As we explain in \S\ref{sec:motivation}, the resolution $\cR$ in Theorem~\ref{thm:main} should be considered as a natural extension of Beilinson's resolution over projective space and similar resolutions due to Buchdahl for Hirzebruch surfaces \cite{Buchdahl87}, Canonaco--Karp for weighted projective stacks \cite{CK08}, and Kapranov for quadrics and flag varieties \cite{Kapranov88}. See \cite[\S4]{BE21} for a related idea, where a resolution of the diagonal---with terms given by infinite direct sums of line bundles---is obtained for any projective toric stack.

We apply Theorem~\ref{thm:main} to make progress on a conjecture concerning virtual resolutions in commutative algebra, a notion introduced by Berkesch--Erman--Smith in \cite{BES20}. We recall that a \emph{virtual resolution} of a graded module $M$ over the Cox ring $S$ of a toric variety $Y$ is a complex $F$ of graded free $S$-modules such that the associated complex of sheaves $\widetilde{F}$ on $Y$ is a locally free resolution of $\widetilde{M}$. The following conjecture predicts a version of Hilbert's Syzygy Theorem for virtual resolutions:

\begin{conjecture}[{\cite[Question~6.5]{BES20}}]\label{conj:BES}
 If $Y$ is a smooth toric variety with Cox ring $S$ and irrelevant ideal $B$, and $M$ is a finitely generated, $B$-saturated $S$-module, then $M$ admits a virtual resolution of length at most $\dim(Y)$.
\end{conjecture}

This conjecture was proven by Berkesch--Erman--Smith for products of projective spaces in \cite{BES20} (see also \cite[Cor.~2.14]{EES15}) and for monomial ideals in Cox rings of smooth toric varieties by Yang in \cite{Yang21}. As a consequence of Theorem~\ref{thm:main}, we prove the following:

\begin{corollary}\label{cor:rank2}
  Conjecture~\ref{conj:BES} holds for smooth projective toric varieties of Picard rank 2.
\end{corollary}

Theorem~\ref{thm:main} also yields a new proof, in the case of smooth projective toric varieties of Picard rank 2, of the following conjecture of Orlov:

\begin{conjecture}[{\cite[Conj.~10]{Orlov09}}]\label{conj:orlov}
  Let $Y$ be a smooth quasi-projective scheme. The Rouquier dimension of the bounded derived category $\Db(Y)$ is equal to $\dim(Y)$.
\end{conjecture}

We refer the reader to the original paper \cite{Rou08} of Rouquier for background on his notion of dimension for triangulated categories. Since the resolution of the diagonal $\cR$ in Theorem~\ref{thm:main} has length $\dim{X}$, and each term $\cR_i$ is a sum of box products of vector bundles on $X$, it is an immediate consequence of \cite[Prop.~7.6]{Rou08} that Theorem~\ref{thm:main} implies Conjecture~\ref{conj:orlov} for smooth projective toric varieties of Picard rank 2. Conjecture~\ref{conj:orlov} was first proven in this case by Ballard--Favero--Katzarkov~\cite[Cor.~5.2.6]{BFK19} using an entirely different approach: they first observe that the conjecture holds for a smooth projective Picard rank 2 toric variety that is weakly Fano, and then they apply descent along admissible subcategories. See the discussion beneath \cite[Conj.~1.1]{BC21} for a list of known cases of Conjecture~\ref{conj:orlov}.

We also apply Theorem~\ref{thm:main} to obtain a splitting criterion for vector bundles on smooth projective toric varieties of Picard rank 2. A famous result of Horrocks~\cite{Horrocks64} states that if a vector bundle on projective space has no intermediate cohomology, then it splits as a sum of line bundles. This splitting criterion has been generalized in many different directions: for instance, to products of projective spaces \cite{CM05,EES15,Schreyer22}, to Grassmannians and quadrics \cite{Ott89}, and to rank 2 vector bundles on Hirzebruch surfaces \cite{FM11, Yas15}, among others. Our splitting criterion for smooth projective toric varieties of Picard rank 2 extends Eisenbud--Erman--Schreyer's for products of projective spaces \cite[Thm.~7.2]{EES15}.

To state the result, we must fix some notation. Given $(a,b), (c,d) \in \ZZ^2$, we write $(a,b) \le (c,d)$  if $a \le c$ and $b \le d$. For a sheaf $\F$ on $X$, let $\g(\F)$ denote its cohomology table:
\[ \g(\F) = \left (\dim_\kk H^i(X, \F(a,b))\right)_{i \ge 0, \text{ } (a,b) \in \ZZ^2}. \]
Here, as in Theorem~\ref{thm:main}, we identify $\Pic{X}$ with $\ZZ^2$ via the choice of basis described in Convention~\ref{conv:basis} below. Our splitting criterion is as follows:

\begin{theorem}\label{thm:splitting}
Let $\E$ be a vector bundle on $X$. Suppose we have
\[ \g(\E) = \sum_{i = 1}^t \g(\OO(b_i, c_i)^{m_i}). \]
If $(b_t, c_t) \le (b_{t-1}, c_{t-1}) \le \cdots \le(b_1, c_1)$, then $\E \cong \bigoplus_{i = 1}^t \OO(b_i, c_i)^{m_i}$.
\end{theorem}

Our proof of Theorem~\ref{thm:splitting} uses a Beilinson-type spectral sequence built from the resolution of the diagonal in Theorem~\ref{thm:main}. This approach is similar to the technique used by Fulger--Marchitan to obtain a splitting criterion for rank 2 vector bundles on Hirzebruch surfaces \cite{FM11}, which involves a Beilinson-type spectral sequence built from Buchdahl's resolution of the diagonal for Hirzebruch surfaces \cite{Buchdahl87}. See also Aprodu--Marchitan's triviality criterion for vector bundles on Hirzebruch surfaces \cite[Thm.~2]{AM11}, whose proof also involves a Beilinson-type spectral sequence.

When $X = \PP^r \times \PP^s$, Theorem~\ref{thm:splitting} recovers (a special case of) \cite[Thm.~7.2]{EES15}. We note that the nef cone of $X$ is given by $\on{Nef}{X} = \{\OO(a, b) \in \Pic{X} \colon a, b \ge 0\}$, and so $(a, b) \le (c, d)$ if and only if the line bundle $\OO(c - a, d - b)$ is nef. Theorem~\ref{thm:splitting} therefore adds a new wrinkle that is not present on products of projective spaces: we require the twists $(b_i, c_i)$ to be ordered with respect to the nef cone, rather than the effective cone. This distinction is invisible in \cite[Thm.~7.2]{EES15}, as the nef and effective cones of a product of projective spaces coincide.

Motivated by the applications of Theorem~\ref{thm:main} described above, we pose the following:

\begin{question}\label{question}
  Can Theorem~\ref{thm:main} be generalized to any smooth projective toric variety $X$?
\end{question}

The difficulty in generalizing Theorem~\ref{thm:main} is in choosing an appropriate toric fiber bundle $E$ over $X \times X$. A positive answer to Question~\ref{question} would immediately resolve the projective case of Conjecture~\ref{conj:BES} and imply a large swath of new cases of Conjecture~\ref{conj:orlov}.

\subsection*{Overview}

We begin in \S\ref{sec:motivation} by constructing a resolution of the diagonal over $\PP^n$ as the pushforward of a Koszul complex over a certain projective bundle, which illustrates our main approach. We prove Theorem~\ref{thm:main} and Corollary~\ref{cor:rank2} in \S\ref{sec:picard-rank-2}, and we prove Theorem~\ref{thm:splitting} in \S\ref{sec:horrocks}.

\subsection*{Acknowledgments}

We thank Matthew Ballard, Christine Berkesch, Lauren Cranton Heller, David Eisenbud, Daniel Erman, David Favero, and Frank-Olaf Schreyer for valuable conversations. We also thank the referee for many helpful comments that improved this paper. The computer algebra system \texttt{Macaulay2}~\cite{M2} provided indispensable assistance throughout this project. The second author was partially supported by the NSF grant DMS-2001101.

\section{Warm-up: the case of $\PP^n$}\label{sec:motivation}

Throughout the paper, we work over a base field $\kk$. Let $\cT_{\PP^n}$ denote the tangent bundle on $\PP^n$ and $\cW$ the vector bundle $\OO_{\PP^n}(1)\boxtimes\cT_{\PP^n}(-1)$ on $\PP^n\times\PP^n$. There is a canonical section $s\in H^0(\PP^n\times\PP^n, \cW)$ whose vanishing cuts out the diagonal in $\PP^n\times\PP^n$ (see \cite[\S8.3]{Huybrechts2006}). The Koszul complex associated to $s$ yields Beilinson's resolution of the diagonal
\[ 0 \from \OO_\Delta \from \OO_{\PP^n\times\PP^n} \from \Lambda^1 \cW^\vee \from \cdots \from \Lambda^n \cW^\vee \from 0. \]
In this section, we construct another resolution of the diagonal sheaf on $\PP^n\times\PP^n$, whose terms are direct sums of line bundles (cf.~\cite[Rem.~3.3]{CK08}). We explain in Remark~\ref{rem:Pn}(3) a sense in which this resolution resembles Beilinson's. As discussed in the introduction, our approach is similar to Weyman's ``geometric technique'' \cite[\S5]{Weyman2003}. In \S\ref{sec:picard-rank-2}, we explain how the approach in this section extends to smooth projective toric varieties of Picard rank 2.

Let $E$ denote the projective bundle $\PP(\OO\oplus\OO(-1,1))$ on $\PP^n\times\PP^n$ and let $\pi\colon E\to\PP^n\times\PP^n$ be the canonical map. The projective bundle $E$ is a toric variety with $\ZZ^3$-graded Cox ring $S_E = \kk[x_0,\dots,x_n, y_0,\dots,y_n,u_0,u_1]$, where $\deg(x_i)=(1,0,0)$, $\deg(y_i)=(0,1,0)$, $\deg(u_0)=(1,-1,1)$, and $\deg(u_1)=(0,0,1)$. Set $\a_i = u_1x_i - u_0y_i$ for all $i$; the intuition here is that $u_0$ and $u_1$ are homogenizing variables for the non-homogeneous equations $x_i - y_i$. Let $\cK$ denote the Koszul complex on $\a_0, \dots, \a_n$, considered as a complex of sheaves on $E$, and set $\cV = \OO(-1, 0, 0)^{n+1}$. Twisting $\cK$ by $\OO(0,0,n)$ yields a complex of the form
\[
\OO(0,0,n)  \from (\Lambda^1\cV)(0,0,n-1) \from \cdots \from \Lambda^n \cV \from (\Lambda^{n+1}\cV)(0,0,-1).
\]
Using \cite[Ch.~III, Ex.~8.4(a)]{Hartshorne1977} and the projection formula, $\cR = \pi_*\cK(0,0,n)$ has the form
\begin{equation}\label{eq:koszul}
  \Sym^n\cQ \from \Lambda^1\cP \otimes \Sym^{n-1}\cQ \from \cdots \from \Lambda^{n-1}\cP \otimes \Sym^1\cQ \from \Lambda^n\cP,
\end{equation}
where $\cP = \OO(-1, 0)^{n+1}$ and $\cQ = \OO \oplus \OO(-1, 1)$. Notice that applying $\pi_*$ to the $n+1^{\th}$ term $(\Lambda^{n+1}\cV)(0,0,-1)$ of $\cK(0,0,n)$ gives 0, hence the complex \eqref{eq:koszul}  has length $n$.
\begin{proposition}\label{prop:EN}
  The complex $\cR$ is a resolution of the diagonal sheaf on $\PP^n \times \PP^n$. Moreover, the complex $\cR$ is isomorphic to (the sheafification of) the $n^{\th}$ symmetric power of the complex
  \begin{equation}\label{eq:section}
    S(-1, 1) \oplus S  \xgets{\begin{pmatrix} -y_0 & -y_1 & \cdots & -y_n \\ x_0 & x_1 & \cdots & x_n \end{pmatrix}} S(-1, 0)^{n+1},
  \end{equation}
  concentrated in homological degrees $0$ and $1$, where $S$ denotes the Cox ring of $\PP^n \times \PP^n$.
\end{proposition}
\begin{proof}
  One can use a slight variation of the proof of Theorem~\ref{thm:main} below to show that $\cR$ is a resolution of the diagonal. As for the second statement: let $K$ denote the Koszul complex on the regular sequence $\a_0,\dots,\a_n$, considered as a complex of $S_E$-modules. Let $R$ be the complex of $S$-modules given by $K(0,0,n)_{(*,*,0)}$. Since $K$ is exact in positive homological degrees, $R$ is as well. It follows from the description of $\cR$ in \eqref{eq:koszul} that $R$ sheafifies to $\cR$. Let $R'$ denote the $n^{\th}$ symmetric power of \eqref{eq:section}. We observe that $R'$ has exactly the same terms as $R$. The complex $R'$ is precisely the generalized Eagon--Northcott complex of type $\EuScript{C}^n$, as defined in \cite[A2.6]{Eisenbud1995}, associated to the map \eqref{eq:section}. It therefore follows from \cite[Thm.~A2.10(c)]{Eisenbud1995} that $R'$ is exact in positive homological degrees. By the uniqueness of minimal free resolutions, we need only check that the cokernels of the first differentials of $R$ and $R'$ are isomorphic, and this can be verified by direct computation.
\end{proof}

We now compute a well-known example using this approach (cf.~\cite[Ex.~5.2]{King97}).

\begin{example}
  Suppose $n = 2$. The monomials in the $u_i$'s give bases for the symmetric powers of $\cQ$, and the exterior monomials in the $\a_i$'s give bases for the terms of $\cK$, which correspond to the exterior powers of $\cP$. Hence, we may index the summands of \eqref{eq:koszul} by monomials in $u_0, u_1, \a_0, \a_1, \a_2$. With this in mind, the complex \eqref{eq:koszul} has terms
  \[
  \ub{\OO(-2,2)  }_{u_0^2}  \oplus
  \ub{\OO(-1,1)  }_{u_0u_1} \oplus
  \ub{\OO        }_{u_1^2}  \xgets{\del_1}
  \ub{\OO(-2,1)^3}_{\a_0u_0, \a_1u_0, \a_2u_0} \oplus
  \ub{\OO(-1,0)^3}_{\a_0u_1, \a_1u_1, \a_2u_1} \xgets{\del_2}
  \ub{\OO(-2,0)^3}_{\a_0\a_1, \a_0\a_2, \a_1\a_2}
  \]
  and differentials
  \[ \del_1 = \begin{pmatrix}
    -y_0 & -y_1  & -y_2 &  0   &  0   &  0   \\
    x_0  &  x_1  &  x_2 & -y_0 & -y_1 & -y_2 \\
    0    &  0    &  0   &  x_0 &  x_1 &  x_2
  \end{pmatrix}
  \quad \text{and} \quad
  \del_2 = \begin{pmatrix}
    y_1  &  y_2 &  0   \\
    -y_0 &  0   &  y_2 \\
    0    & -y_0 & -y_1 \\
    -x_1 & -x_2 &  0   \\
    x_0  &  0   & -x_2 \\
    0    &  x_0 &  x_1
  \end{pmatrix}. \]
\end{example}

\begin{remark}\label{rem:Pn}
  We conclude this section with the following observations:
  \begin{enumerate}
  \item We have $\on{rank} \cR_i = \on{rank} \cR_{n - i}$, just as in Theorem~\ref{thm:main}.
  \item The resolutions in Theorem~\ref{thm:main} cannot arise as symmetric powers of complexes, in general; this follows immediately from rank considerations.
  \item Let us explain a sense in which our resolution $\cR$ is modeled on Beilinson's resolution of the diagonal. Consider the external tensor product of $\OO(1)$ with the Euler sequence:
    $0 \from \OO(1)\boxtimes\cT(-1) \from \OO(1,0)^{n+1}
    \xgets{\begin{pmatrix} y_0 & \cdots & y_n\end{pmatrix}^T} \OO(1,-1) \from 0.$
    Letting $\cC$ denote the subcomplex $\OO(1,0)^{n+1} \from \OO(1,-1)$ concentrated in degrees 0 and 1, there is a quasi-isomorphism $\cC \xto{\simeq} \OO(1) \boxtimes \cT(-1)$. The morphism \( s\colon\OO\xto{\begin{pmatrix} x_0 & \cdots & x_n \end{pmatrix}^T} \C, \) where $\OO$ lies in degree 0, gives a hypercohomology class in $\mathbb{H}^0(\PP^n\times\PP^n, \C)$, which is isomorphic to $H^0(\PP^n \times \PP^n, \OO(1) \boxtimes \cT(-1))$. By Proposition~\ref{prop:EN}, the $n^{\th}$ symmetric power of the dual of $s$, i.e. the $n^{\th}$ Koszul complex of the dual of $s$ \cite[Definition 2.3]{Kock01}, is isomorphic to the resolution $\cR$. In short: the resolution $\cR$ is a Koszul complex on a section of $\OO(1) \boxtimes \cT(-1)$, just like Beilinson's resolution.
  \end{enumerate}
\end{remark}

\section{Smooth Projective Toric Varieties of Picard rank 2}\label{sec:picard-rank-2}
In this section, we extend the construction in \S\ref{sec:motivation} and prove the main theorem. Let $X$ denote the projective bundle \( \PP(\OO \oplus \OO(a_1) \oplus \cdots \oplus \OO(a_s)) \) over $\PP^r$, where $a_1 \le \cdots \le a_s.$ As discussed in \cite[\S7.3]{CLS2011}, the fan $\Sigma_X \subseteq \ZZ^{r + s}$ of $X$ has $r + s + 2$ ray generators given by the rows 
of the $(r + s + 2) \times (r + s)$ matrix
\begin{equation}\label{eq:X-rays}
  \begin{footnotesize}
    P = \begin{pmatrix}
    -1 & -1 & \cdots & -1 & a_1 & a_2 & \cdots & a_s \\
    1  & 0  & \cdots & 0  & 0   & 0   & \cdots & 0   \\
    0  & 1  & \cdots & 0  & 0   & 0   & \cdots & 0   \\
    \vdots  & \vdots & & \vdots & \vdots & \vdots & & \vdots \\
    0  & 0  & \cdots & 1  & 0   & 0   & \cdots & 0   \\
    0  & 0  & \cdots & 0  & -1  & -1  & \cdots & -1  \\
    0  & 0  & \cdots & 0  & 1   & 0   & \cdots & 0   \\
    0  & 0  & \cdots & 0  & 0   & 1   & \cdots & 0   \\
    \vdots  & \vdots & & \vdots & \vdots & \vdots & & \vdots \\
    0  & 0  & \cdots & 0  & 0   & 0   & \cdots & 1   \\
  \end{pmatrix}
  = \begin{pmatrix}
  \rho_0 \\ \rho_1 \\ \rho_2 \\ \vdots \\ \rho_r \\ \sigma_0 \\ \sigma_1 \\ \sigma_2 \\ \vdots \\ \sigma_s
  \end{pmatrix}
  \end{footnotesize}
\end{equation}
and maximal cones generated by collections of rays of the form
\[ \{\rho_0, \dots, \widehat{\rho_i}, \dots, \rho_r, \sigma_0, \dots,  \widehat{\sigma_j}, \dots, \sigma_s\}. \]

\begin{convention}\label{conv:basis}
Throughout the paper, we equip $\Pic{X} \cong \coker(P) \cong \ZZ^2$ with the basis given by the divisors corresponding to $\rho_0$ and $\sigma_0$. With this choice of basis, we may view the Cox ring of $X$ as the $\ZZ^2$-graded ring $\kk[x_0, \dots, x_r, y_0, \dots, y_s]$ whose variables have degrees given by the columns of the matrix
\[ A = \begin{pmatrix}
1 & 1 & \cdots & 1 & 0 & -a_1 & \cdots & -a_s \\
0 & 0 & \cdots & 0 & 1 & 1 & \cdots & 1 \\
\end{pmatrix}. \]
A main reason we use this convention is that it is also used by the function \texttt{kleinschmidt} in \texttt{Macaulay2}, which produces any smooth projective toric variety of Picard rank 2 as an object of type \texttt{NormalToricVariety}.
\end{convention}

\subsection{Vanishing of sheaf cohomology}\label{sec:vanishing}

We will need a calculation of the cohomology of a line bundle on $X$:

\begin{proposition}\label{prop:vanishing}
  Let $\E$ be the vector bundle $\OO\oplus\OO(a_1)\oplus\cdots\oplus\OO(a_s)$ on $\PP^r$, where $a_1 \le \cdots \le a_s$, so that $X = \PP(\E)$. Write $m = \sum_{i=1}^s a_i$, and consider a line bundle $\OO(k,\ell)$ on $X$. For each $0\le j\le r+s$, we have:
  \[ H^j(X, \OO(k, \ell)) \cong \begin{cases}
    H^j    (\PP^r, \OO_{\PP^r}(k) \otimes \Sym^{\ell}(\E)),             & \ell \ge 0; \\
    H^{j-s}(\PP^r, \OO_{\PP^r}(k-m) \otimes \Sym^{-\ell-s-1}(\E)^\vee), & \ell \le -s -1; \\
    0, & \text{otherwise.}
  \end{cases} \]
\end{proposition}
\begin{proof}
  Let $\pi\colon X \to \PP^r$ denote the projective bundle map. It follows from a well-known calculation (see e.g.~\cite[4.5(e)]{TT90}) and the projection formula that
  \[ \RR^i\pi_*(\OO(k, \ell)) = \begin{cases}
    \OO_{\PP^r}(k) \otimes \Sym^{\ell}(\E), & i = 0; \\
    \OO_{\PP^r}(k-m) \otimes \Sym^{-\ell-s-1}(\E)^\vee, & i = s; \\
    0, & 0 < i < s.
  \end{cases} \]
  The conclusion follows from the observation that the second page of the Grothendieck spectral sequence
  \[ E^{p,q}_2 = H^p\left(\PP^r, \RR^q \pi_*(\OO(k, \ell))\right) \Rightarrow H^{p + q}(X, \OO(k, \ell)) \]
  collapses to row $q = 0$ when $\ell \ge 0$ and to row $q = s$ when $\ell \le -s-1$. 
\end{proof}

The following result is an immediate consequence of Proposition~\ref{prop:vanishing}. It will play a key role in the proofs of Theorems~\ref{thm:main} and~\ref{thm:splitting}.

\begin{corollary}\label{cor:vanishing}
  Let $X$ be the projective bundle $\PP(\OO\oplus\OO(a_1)\oplus\cdots\oplus\OO(a_s))$ over $\PP^r$ as above, where $a_1\leq\cdots\leq a_s$. Write $m = \sum_{i=1}^s a_i$, and consider a line bundle $\OO(k, \ell)$ on $X$.
  \begin{enumerate}
  \item We have:
    \begin{enumerate}
    \item $H^i(X, \OO(k, \ell)) = 0$ if $i \notin \{0, r, s, r+s\}$.
    \item $H^0(X, \OO(k, \ell)) = 0$ if and only if $\ell < 0$ or $k + a_s \ell < 0$.
    \item If $r \ne s$ then
      \begin{enumerate}
        \item $H^r(X, \OO(k, \ell)) = 0$ if and only if $-r-1 < k$ or $\ell < 0$, and
        \item $H^s(X, \OO(k, \ell)) = 0$ if and only if $-s-1 < \ell$ or $k < m$.
          \end{enumerate}
    \item If $r = s$ then $H^r(X, \OO(k, \ell)) = 0$ if and only if both of the following hold:
      \begin{enumerate}
      \item $-r-1 < k$ or $\ell < 0$, and
      \item $-s-1 < \ell$ or $k < m$.
      \end{enumerate}
    \item Lastly, $H^{r + s}(X, \OO(k, \ell)) = 0$ if and only if either of the following hold:
      \begin{enumerate}
      \item $-r-1 - a_s(\ell + s + 1) + m < k$, or
      \item $-s-1 < \ell$;
      \end{enumerate}
    \end{enumerate}
  \item In particular, the line bundle $\OO(k, \ell)$ is acyclic ($H^i(X, \OO(k, \ell)) = 0$ for $i > 0$) if and only if one of the following holds:
    \begin{itemize}
    \item[(a)] $-s-1 < \ell < 0$,
    \item[(b)] $-r-1 < k$ and $0\le\ell$,
    \item[(c)] $-r-1 - a_s(\ell + s + 1) + m < k < m$ and $\ell \le -s-1$.
    \end{itemize}
  \end{enumerate}
\end{corollary}

\begin{remark}
Conditions (1b) and (1e) are Serre dual to one another. Ditto for the two conditions in (1c), as well as the conditions (i) and (ii) in (1d). These calculations are surely well-known; see, for instance, \cite[Prop.~3.9]{LM11} for a criterion for acyclicity of line bundles on toric varieties. We refer the reader to \cite[Ex.~3.14]{BE21} for a depiction of the regions of $\ZZ^2$ where each $H^i(X, \OO(k, \ell))$ vanishes for the Hirzebruch surface $X=\PP(\OO_{\PP^1}\oplus\OO_{\PP^1}(3))$.
\end{remark}

\subsection{Toric fiber bundles}

Let $E$ and $Y$ be smooth projective toric varieties of dimensions $d_E$ and $d_Y$ associated to fans $\Sigma_E$ and $\Sigma_Y$. Let $\bar\pi\colon\ZZ^{d_E} \to \ZZ^{d_Y}$ be a $\ZZ$-linear surjection that is compatible with the fans $\Sigma_E$ and $\Sigma_Y$, in the sense of \cite[Def.~3.3.1]{CLS2011}, so that it induces a morphism $\pi\colon E \to Y$.
We denote by $F$ the toric variety associated to the fan $\Sigma_F =\{ \sigma \in \Sigma_E \colon\sigma \subseteq \ker(\bar\pi)_\mathbb{R} \}$, and write $d_F = \dim{F}$. Let us assume that the fan $\Sigma_E$ is \emph{split} by the fans $\Sigma_Y$ and $\Sigma_F$, in the sense of \cite[Def.~3.3.18]{CLS2011}. In this case, the map $\pi\colon E \to Y$ is a fibration with fiber $F$; see \cite[Thm.~3.3.19]{CLS2011}.

Writing the Cox rings of $Y$ and $F$ as $S_Y = \kk[x_1, \dots, x_{n_1}]$ and $S_F = \kk[u_1, \dots, u_{n_2}]$, the Cox ring of $E$ has the form $S_E = \kk[x_1, \dots, x_{n_1}, u_1, \dots, u_{n_2}]$. We have presentations $P_Y\colon\ZZ^{d_Y} \to \ZZ^{n_1}$ and $P_F\colon\ZZ^{d_F} \to \ZZ^{n_2}$ of $\Pic{Y}$ and $\Pic{F}$ whose rows are given by the ray generators of $\Sigma_Y$ and $\Sigma_F$, respectively. The analogous presentation of $\Pic{E}$ is of the form
\[ \begin{pmatrix} P_Y & Q \\ 0 & P_F\end{pmatrix} \]
for some $n_1 \times d_F$ matrix $Q$. One may use this presentation to equip $S_E$ with a $\ZZ^e \oplus \ZZ^f$-grading such that $\deg_{S_E}(x_i) = (\deg_{S_Y}(x_i), 0)$, and $\deg_{S_E}(u_i) = (t_i,\deg_{S_F}(u_i))$ for some $t_i \in \ZZ^e$.

\begin{lemma}[cf.~\cite{Hartshorne1977} Ch.~III, Ex.~8.4(a)]\label{lem:cohomology}
  Let $\cL = \OO_E(b_1, \dots, b_{e}, c_1, \dots, c_{f})$, and let $\cB$ be a $\kk$-basis of $H^0(F, \OO_F(c_1, \dots, c_{f}))$ given by monomials in $S_F$. Given $m \in \cB$, denote its degree in $S_E$ by $(d_1^m, \dots, d_e^m, c_1, \dots, c_f)$. We have \( \pi_*(\cL) \cong \bigoplus_{m \in \cB} \OO_Y(b_1 - d_1^m, \dots, b_{e} - d_{e}^m)\).
  Moreover, if $H^i(F, \OO_F(c_1, \dots, c_{f})) = 0$, then $\RR^i\pi_*(\cL) = 0$.
\end{lemma}
\begin{proof}
  Let $g\colon\bigoplus_{m \in \cB} \OO_Y(b_1 - d_1^m, \dots, b_{e} - d_{e}^m) \to \pi_*(\cL)$ be the morphism given on the component corresponding to $m \in \cB$ by multiplication by $m$. Let $U$ be an affine open subset of $Y$ over which the fiber bundle $E$ is trivializable; abusing notation slightly, we denote by $\pi$ the map $\pi^{-1}(U) \to U$ induced by $\pi$. To prove the first statement, it suffices to show that the restriction $g_U:\bigoplus_{m \in \cB_i} \OO_U \to  \pi_*(\cL|_U)$
  of $g$ to $U$ is an isomorphism. Without loss of generality, we may assume that $\pi^{-1}(U)  = U \times F$ and that $\pi\colon\pi^{-1}(U) \to U$ is the projection onto $U$. Letting $\g: \pi^{-1}(U) \to F$ denote the projection, we have that $\cL|_U  = \g^*(\OO_F(c_1, \dots, c_{f}))$. Finally, we observe that $g_U$ coincides with the base change isomorphism
  \( \bigoplus_{m \in \cB} \OO_U = \OO_U \otimes_\kk H^0(F, \OO_F(c_1, \dots, c_{f})) \xto{\cong} \pi_*(\g^*(\OO_F(c_1, \dots, c_{f})) = \pi_*(\cL|_U). \)
  As for the last statement: it suffices to observe that, by base change,
  \( \RR^i\pi_*(\cL|_U) \cong \OO_U \otimes_\kk H^i(F, \OO_F(c_1, \dots, c_f)) = 0 \).
\end{proof}


\subsection{Constructing the resolution of the diagonal}\label{sec:bundle}

Let $X$ be as defined at the beginning of this section. We will construct our resolution of the diagonal for $X$ as the pushforward of a certain Koszul complex on a fibration $E$ over $X \times X$ whose fiber is the Hirzebruch surface $\FF_{a_s}$. We begin by constructing the fiber bundle $\pi\colon E \to X \times X$. The ray generators of $E$ are given by the rows of the $(2r + 2s + 8) \times (2r + 2s + 2)$ matrix
\begin{equation}\label{eq:raymatrix}
 \left(\begin{array}{c c | c c}
	P & 0 &  v & -w \\
	0 & P & -v &  w \\
	\hline
	0 & 0 & -1 & a_s \\
	0 & 0 &  0 &  1  \\
	0 & 0 &  1 &  0  \\
	0 & 0 &  0 & -1
\end{array}\right),
\end{equation}
where $P$ is as in \eqref{eq:X-rays}, and $v$ (resp. $w$) is the $(r + s + 2) \times 1$ matrix with unique nonzero entry given by a 1 in the first (resp. $(r+2)^{\th}$) position. Notice that the rows in the top-left quadrant of this matrix are the ray generators of $X \times X$, and the rows in the bottom-right quadrant are the ray generators of $\FF_{a_s}$.

Let $\bar\pi\colon\ZZ^{2r+2s+2} \to \ZZ^{2r+2s}$ denote the projection onto the first $2r+2s$ coordinates. We define the cones of $E$ to be those of the form $\g + \g'$, where $\g$ is a cone corresponding to a cone of $\FF_{a_s}$ and is spanned by a subset of the bottom 4 rows of \eqref{eq:raymatrix}, and $\g'$ is a cone spanned by a collection of the top $2r + 2s + 4$ rows of \eqref{eq:raymatrix} such that $\bar\pi_\mathbb{R}(\g')$ is a cone of $X \times X$. By \cite[Thm.~3.3.19]{CLS2011}, the map $\bar\pi$ induces a fibration $\pi\colon E \to X$ with fiber $\FF_{a_s}$.

In order to describe the Cox ring of $E$, first recall the matrix $A$ from Convention~\ref{conv:basis} whose columns are the degrees of the variables of the Cox ring of $X$, and consider the matrices
\[ B = \begin{pmatrix}
  1 & -a_s & 0 & 0 \\
  0 & 1 & 0 & 0\\
\end{pmatrix}
\quad \text{ and } \quad
C = \begin{pmatrix}
  1 & -a_s & 1 & 0 \\
  0 & 1 & 0 & 1\\
\end{pmatrix}. \]
Notice that the columns of $C$ are the degrees of the variables in the Cox ring of $\FF_{a_s}$. We choose a basis of $\Pic{E} \cong \ZZ^6$ so that the degrees of the variables in the Cox ring
\[ S_E = \kk[x_0, \dots x_r, y_0, \dots, y_s, x_0', \dots, x_r', y_0',\dots, y_s', u_0, \dots, u_3] \]
of $E$ are given by the columns of the Gale dual of \eqref{eq:raymatrix}, which is the $6\times(2r+2s+8)$ matrix
\[ \begin{footnotesize}
  \begin{pmatrix}
    A & 0 &  B \\
    0 & A & -B \\
    0 & 0 &  C \\
  \end{pmatrix} =
  \setcounter{MaxMatrixCols}{20}
  \begin{pmatrix}
    1 & \cdots & 1 & 0 & -a_1 & \cdots & -a_s & 0 & \cdots & 0 & 0 &  0   & \cdots &  0   &  1 & -a_s & 0 & 0 \\
    0 & \cdots & 0 & 1 &  1   & \cdots &  1   & 0 & \cdots & 0 & 0 &  0   & \cdots &  0   &  0 &  1   & 0 & 0 \\
    0 & \cdots & 0 & 0 &  0   & \cdots &  0   & 1 & \cdots & 1 & 0 & -a_1 & \cdots & -a_s & -1 &  a_s & 0 & 0 \\
    0 & \cdots & 0 & 0 &  0   & \cdots &  0   & 0 & \cdots & 0 & 1 &  1   & \cdots &  1   &  0 & -1   & 0 & 0 \\
    0 & \cdots & 0 & 0 &  0   & \cdots &  0   & 0 & \cdots & 0 & 0 &  0   & \cdots &  0   &  1 & -a_s & 1 & 0 \\
    0 & \cdots & 0 & 0 &  0   & \cdots &  0   & 0 & \cdots & 0 & 0 &  0   & \cdots &  0   &  0 &  1   & 0 & 1 \\
  \end{pmatrix}.
\end{footnotesize} \]

Let $K$ be the Koszul complex corresponding to the regular sequence $\a_0,\dots,\a_r,\b_0,\dots,\b_s$ given by the homogeneous binomials
\begin{align*}
  \a_i &= u_2x_i - u_0x_i' \quad \quad \quad \quad \;\; \text{ for } 0 \le i \le r \; \text{ and} \\
  \b_i &= u_3y_i - u_0^{a_s-a_i}u_1u_2^{a_i}y_i' \quad \text{ for } 0 \le i \le s \;\; (a_0 \coloneqq 0)
\end{align*}
in the Cox ring $S_E$. Observe that $\deg(\a_i) = (1,0,0,0,1,0)$ and $\deg(\b_i) = (-a_i,1,0,0,0,1)$. Here, we are using that the columns of $B$ span the effective cone of $X$ to homogenize the relations $x_i - x_i'$ and $y_i - y_i'$. Denote by $\cK$ the complex of sheaves on $E$ corresponding to $K$. The following proposition shows that $\cK$ twisted by $\OO_E(0,0,0,0,r,s)$ is $\pi_*$-acyclic. 

\begin{proposition}\label{prop:higherdirect}
  The higher direct images $\RR^i\pi_*(\cK(0,0,0,0,r,s))$ vanish for $i>0$.
\end{proposition}
\begin{proof}
  It suffices to show that $\RR^i\pi_*(\cK_j(0,0,0,0,r,s)) = 0$ for $i>0$ and all $j$. Each term of $\cK(0,0,0,0,r,s)$ is a direct sum of line bundles of the form $\OO_E(a,b,0,0,k,\ell))$ for some $a,b\in\ZZ$, $-1\le k\le r$, and $-1\le\ell\le s$. By Lemma~\ref{lem:cohomology}, we need only show that $H^i(\FF_{a_s}, \OO(k,\ell)) = 0$ for $i>0$ and such $k$ and $\ell$, which follows from Corollary~\ref{cor:vanishing}(2)(a-b).
\end{proof}

Let $S$ denote the Cox ring of $X \times X$ and $R$ the complex of graded $S$-modules given by the subcomplex $K(0,0,0,0,r,s)_{(\ast,\ast,\ast,\ast,0,0)}$ of the Koszul complex $K$ twisted by $S_E(0,0,0,0,r,s)$. We will show that $R$ satisfies the requirements of Theorem~\ref{thm:main}. Observe that, by Lemma~\ref{lem:cohomology}, one can alternatively construct $R$ by applying the twisted global sections functor:
\[ R = \bigoplus_{\cL\in\Pic(X\times X)} H^0(X\times X,\cL\otimes\pi_*\cK(0,0,0,0,r,s)). \]
In particular, writing $\cR$ for the complex of sheaves on $X\times X$ corresponding to $R$, we have $\cR\cong\pi_*\cK(0,0,0,0,r,s)$. Note that Proposition~\ref{prop:higherdirect} implies that $\pi_*\cK(0,0,0,0,r,s)$ is quasi-isomorphic to $\RR\pi_*(\cK(0,0,0,0,r,s))$.

Before discussing some examples, we must establish a bit of notation:

\begin{notation}\label{not:monomials}
Let $S_F = \kk[u_0,u_1,u_2,u_3]$ denote the Cox ring of the Hirzebruch surface $\FF_{a_s}$, equipped with the $\ZZ^2$-grading so that the degrees of the variables correspond to the columns of the matrix $C$ above. Given $i,j\in\ZZ$, let $M_{i,j}$ denote the set of monomials in $S_F$ of degree $(i,j)$. For $m\in M_{i,j}$, let $(d^m_1, d^m_2, d^m_3, d^m_4) \in \ZZ^4$ denote the first four coordinates of the degree of $m$ as an element of the $\ZZ^6$-graded ring $S_E$; notice that $d_3^m = -d_1^m$, and $d_4^m = -d_2^m$.
\end{notation}

\begin{example}\label{ex:firstdiff}
  Let us compute the first differential in $R$. Using the notation above, we have
  \begin{align*}
    R_0    &= \bigoplus_{m \in M_{r,s}} S(-d_1^m, -d_2^m,d_1^m, d_2^m) \cdot m, \quad \text{and} \quad
    R_1     = R^\a_1 \oplus R^\b_1, \quad \text{where} \\
    R_1^\a &= \bigoplus_{i = 0}^r \bigoplus_{m \in M_{r-1, s}} S(-d_1^m - 1, -d_2^m, d_1^m, d_2^m) \cdot \a_i m, \\
    R_1^\b &= \bigoplus_{i = 0}^s \bigoplus_{m \in M_{r, s-1}} S(-d_1^m + a_i, -d_2^m - 1, d_1^m, d_2^m) \cdot \b_i m.
  \end{align*}
  Here, the decorations $``\cdot m"$ in our description of $R_0$ are just for bookkeeping, and similarly for the $``\cdot \a_i m"$ and $``\cdot \b_i m"$ in $R_1$.
 Viewing the differential $\del_1\colon R_1 \to R_0$ as a matrix with respect to the above basis, the column corresponding to $\a_i m$ has exactly two nonzero entries: an entry of $x_i$ corresponding to the monomial $u_2m \in M_{r,s}$ and an entry of $-x_i'$ corresponding to $u_0m \in  M_{r,s}$. Similarly, the column corresponding to $\b_im$ has exactly two nonzero entries: an entry of $y_i$ corresponding to $u_3m$ and an entry of $-y_i'$ corresponding to $u_0^{a_s - a_i}u_1u_2^{a_i}m$. That is, the matrix $\del_1$ has the following form:
 \begin{footnotesize} \[
   \begin{pNiceMatrix}[last-col,last-row,nullify-dots,columns-width=1mm]
          &  0    &        &  0    &        & \Vdots \\
          & -x_i' &        &  0    &        & u_0m   \\
          &  0    &        &  0    &        & \Vdots \\
          &  0    &        & -y_i' &        & u_0^{a_s - a_i}u_1u_2^{a_i} m \\
   \cdots\hspace*{-5mm} &  0    & \hspace*{-5mm}\cdots\hspace*{-5mm} &  0    & \hspace*{-5mm}\cdots & \Vdots \\
          &  x_i  &        &  0    &        & u_2m   \\
          &  0    &        &  0    &        & \Vdots \\
          &  0    &        &  y_i  &        & u_3m   \\
          &  0    &        &  0    &        & \Vdots \\
          & \a_im & \hspace*{-5mm}\cdots\hspace*{-5mm} & \b_im &        &
   \end{pNiceMatrix}.
 \] \end{footnotesize}
 
\end{example}

\begin{example}\label{ex:hirzebruch}
Suppose $X$ is the Hirzebruch surface of type $a$, i.e. the projective bundle $\PP(\OO\oplus\OO(a))$ over $\PP^1$. We have $r=s=1$ and $a_1=a$. The Koszul complex $K$ on $\a_0, \a_1, \b_0, \b_1$, twisted by $(0,0,0,0,1,1)$, looks like:
  \begin{footnotesize}
    \begin{align*}
      \ub{S_E(0,0,0,0,1,1)}_1 &
      \from  \ub{S_E( -1, 0,0,0, 0, 1)^2}_{\a_0,\;\a_1}
      \oplus \ub{S_E(  0,-1,0,0, 1, 0)  }_{\b_0}
      \oplus \ub{S_E(  a,-1,0,0, 1, 0)  }_{\b_1} \\ &
      \from  \ub{S_E( -2, 0,0,0,-1, 1)  }_{\a_0\a_1}
      \oplus \ub{S_E( -1,-1,0,0, 0, 0)^2}_{\a_0\b_0,\;\a_1\b_0}
      \oplus \ub{S_E(a-1,-1,0,0, 0, 0)^2}_{\a_0\b_1,\;\a_1\b_1} \\ & \quad
      \oplus \ub{S_E(  a,-2,0,0, 1,-1)  }_{\b_0\b_1} \\ &
      \from  \ub{S_E(a-2,-1,0,0,-1, 0)  }_{\a_0\a_1\b_1}
      \oplus \ub{S_E( -2,-1,0,0,-1, 0)  }_{\a_0\a_1\b_0}
      \oplus \ub{S_E(a-1,-2,0,0, 0,-1)^2}_{\a_0\b_0\b_1,\;\a_1\b_0\b_1} \\ &
      \from  \ub{S_E(a-2,-2,0,0,-1,-1)  }_{\a_0\a_1\b_0\b_1}.
    \end{align*}
  \end{footnotesize}

  Letting $M_{i, j}$ be as in Notation~\ref{not:monomials} (with $a_s = a$), we have:
  \begin{align*}
    M_{0,0}\;\; &= \{1\},        \hspace{1cm}
    M_{1,0}      = \{u_0, u_2\}, \hspace{1cm}
    M_{0,1}      = \{u_3\} \cup \{u_0^ku_1u_2^{\ell} \colon k + \ell = a \}, \\
    M_{-1,1}    &= \{u_0^ku_1u_2^{\ell} \colon k + \ell = a - 1\}, \\
    M_{1,1}\;\; &= \{u_0u_3, u_2u_3\} \cup \{u_0^ku_1u_2^{\ell} \colon k + \ell = a+1 \}, \\
    M_{i,j}\;\; &= \emptyset \quad \text{for } (i, j) \in \{(1, -1), (-1, 0), (0, -1), (-1, -1)\}.
  \end{align*}

  It follows that the complex $R$ has terms as follows:
  \begin{footnotesize}
    \begin{align*}
      R_0  =& \ub{S( -1,-1, 1,1)  }_{u_0^{a+1} u_1}
      \oplus  \ub{S(  0,-1, 0,1)  }_{u_0^a u_1 u_2} \oplus \cdots
      \oplus  \ub{S(  a,-1,-a,1)  }_{u_1 u_2^{a+1}}
      \oplus  \ub{S( -1, 0, 1,0)  }_{u_0 u_3}
      \oplus  \ub{S(  0, 0, 0,0)  }_{u_2 u_3}, \\
      R_1  =& \ub{S( -1,-1, 0,1)^2}_{\a_0 u_0^a u_1,\; \a_1 u_0^a u_1}
      \oplus  \ub{S(  0,-1,-1,1)^2}_{\a_0 u_0^{a-1} u_1 u_2,\; \a_1 u_0^{a-1} u_1 u_2} \oplus \cdots
      \oplus  \ub{S(a-1,-1,-a,1)^2}_{\a_0 u_1 u_2^a,\; \a_1 u_1 u_2^a}
      \oplus  \ub{S( -1, 0, 0,0)^2}_{\a_0 u_3,\; \a_1 u_3} \\
      \oplus& \ub{S( -1,-1, 1,0)  }_{\b_0 u_0}
      \oplus  \ub{S(a-1,-1, 1,0)  }_{\b_1 u_0}
      \oplus  \ub{S(  0,-1, 0,0)  }_{\b_0 u_2}
      \oplus  \ub{S(a  ,-1, 0,0)  }_{\b_1 u_2}, \\
      R_2  =& \ub{S( -1,-1,-1,1)  }_{\a_0\a_1 u_0^{a-1} u_1}
      \oplus  \ub{S(  0,-1,-2,1)  }_{\a_0\a_1 u_0^{a-2} u_1 u_2} \oplus \cdots
      \oplus  \ub{S(a-2,-1,-a,1)  }_{\a_0\a_1 u_1 u_2^{a-1}}
      \oplus  \ub{S( -1,-1, 0,0)^2}_{\a_0\b_0,\; \a_1\b_0}
      \oplus  \ub{S(a-1,-1, 0,0)^2}_{\a_0\b_1,\; \a_1\b_1}. \\
    \end{align*}
  \end{footnotesize}

  The differentials $\del_1\colon R_0 \gets R_1$ and $\del_2\colon R_1 \gets R_2$ are given, respectively, by the matrices
  \begin{footnotesize}
    \[ \setlength\arraycolsep{2pt}
   \del_1 = \begin{pmatrix}
      -x_0' & -x_1' &  0    &  0    &  0    &  0    & \cdots & 0 & 0 & 0 & 0 &  0    &  0    & -y_0' & 0   &  0    & 0   \\
      x_0   &  x_1  & -x_0' & -x_1' &  0    &  0    & \cdots & 0 & 0 & 0 & 0 &  0    &  0    &  0    & 0   & -y_0' & 0   \\
      0     &  0    &  x_0  &  x_1  & -x_0' & -x_1' & \cdots & 0 & 0 & 0 & 0 &  0    &  0    &  0    & 0   &  0    & 0   \\
      0     &  0    &  0    &  0    &  x_0  &  x_1  & \cdots & 0 & 0 & 0 & 0 &  0    &  0    &  0    & 0   &  0    & 0   \\
      \vdots & \vdots & \vdots & \vdots & \vdots & \vdots & & \vdots & \vdots &
      \vdots & \vdots & \vdots & \vdots & \vdots & \vdots & \vdots & \vdots \\
      0 & 0 & 0 & 0 & 0 & 0 & \cdots & -x_0' & -x_1' &  0    &  0    & 0 & 0 & 0 &  0    & 0 &  0    \\
      0 & 0 & 0 & 0 & 0 & 0 & \cdots &  x_0  &  x_1  & -x_0' & -x_1' & 0 & 0 & 0 & -y_1' & 0 &  0    \\
      0 & 0 & 0 & 0 & 0 & 0 & \cdots &  0    &  0    &  x_0  &  x_1  & 0 & 0 & 0 &  0    & 0 & -y_1' \\
      0 & 0 & 0 & 0 & 0 & 0 & \cdots & 0 & 0 & 0 & 0 & -x_0' & -x_1' & y_0 & y_1 &  0    & 0         \\
      0 & 0 & 0 & 0 & 0 & 0 & \cdots & 0 & 0 & 0 & 0 &  x_0  &  x_1  & 0   & 0   &  y_0  & y_1
    \end{pmatrix},\]
    \[\del_2 =\begin{pmatrix}
      x_1'  &  0    & \cdots & 0     &  y_0' &  0    &  0    &  0    \\
      -x_0' &  0    & \cdots & 0     &  0    &  y_0' &  0    &  0    \\
      -x_1  &  x_1' & \cdots & 0     &  0    &  0    &  0    &  0    \\
      x_0   & -x_0' & \cdots & 0     &  0    &  0    &  0    &  0    \\
      0     & -x_1  & \cdots & 0     &  0    &  0    &  0    &  0    \\
      0     &  x_0  & \cdots & 0     &  0    &  0    &  0    &  0    \\
      \vdots & \vdots & & \vdots & \vdots & \vdots & \vdots\\
      0     &  0    & \cdots &  x_1' &  0    &  0    &  0    &  0    \\
      0     &  0    & \cdots & -x_0' &  0    &  0    &  0    &  0    \\
      0     &  0    & \cdots & -x_1  &  0    &  0    &  y_1' &  0    \\
      0     &  0    & \cdots &  x_0  &  0    &  0    &  0    &  y_1' \\
      0     &  0    & \cdots &  0    & -y_0  &  0    & -y_1  &  0    \\
      0     &  0    & \cdots &  0    &  0    & -y_0  &  0    & -y_1  \\
      0     &  0    & \cdots &  0    & -x_0' & -x_1' &  0    &  0    \\
      0     &  0    & \cdots &  0    &  0    &  0    & -x_0' & -x_1' \\
      0     &  0    & \cdots &  0    &  x_0  &  x_1  &  0    &  0    \\
      0     &  0    & \cdots &  0    &  0    &  0    &  x_0  &  x_1  \\
    \end{pmatrix}. \]
  \end{footnotesize}

  \noindent 
  As predicted by Theorem~\ref{thm:main} parts (2) and (3), the differentials in $R$ are linear; and the ranks of $R_0$, $R_1$, and $R_2$ are $a + 4$, $2a + 8$, and $a + 4$, respectively.
\end{example}

\subsection{A Fourier--Mukai transform}\label{sec:FM}

Let $\pi_1$ and $\pi_2$ denote the projections of $X \times X$ onto $X$, and let $\Phi_\cR$ denote the following Fourier--Mukai transform:
\[\Phi_\cR\colon \Db(X) \xto{\pi_1^*} \Db(X \times X) \xto{\cdot\;\otimes\cR} \Db(X \times X) \xto{\RR\pi_{2*}} \Db(X). \]
We will prove that $\cR$ is a resolution of the diagonal by showing that $\Phi_{\cR}$ is isomorphic to the identity functor, and we will do so by directly exhibiting a natural isomorphism $\Phi_\nu : \Phi_\cR \to \Phi_{\OO_\Delta}$. In fact, we show this by proving that $\Phi_\nu$ induces a quasi-isomorphism on a full exceptional collection. To perform this calculation, we will need an explicit model for the functor $\Phi_{\cR}$, which we present in this section. We refer the reader to \cite[\S8.3]{Huybrechts2006} for further background.

Let $\coh(X)$ denote the category of coherent sheaves on $X$, and suppose $\F_1,\F_2\in\coh(X)$, where $\F_1$ is locally free. By the projection formula and base change, we have canonical isomorphisms
\[ \RR\pi_{2*}(\F_1\boxtimes\F_2) \cong \RR\pi_{2*}\pi^*_1(\F_1)\otimes_{\OO_{X}}\F_2 \cong \RR\Gamma(X, \F_1) \otimes_k \F_2 \]
in $\Db(X)$. Given $\F \in \coh(X)$, we can use this to explicitly compute $\Phi_\cR(\F)$ as follows. Given $\G \in \coh(X)$, let $\cech_\G$ denote the \v{C}ech complex of $\G$ associated to the affine open cover of $X$ arising from the maximal cones in its fan. Consider the following bicomplex, where the horizontal maps are induced by the differentials in $\cR$, the vertical maps are induced by the \v{C}ech differentials, $N$ is the length of $\cR$, and
``$\cL_1 \boxtimes \cL_2 \in \cR_i$'' is shorthand for ``$\cL_1 \boxtimes \cL_2$ is a summand of $\cR_i$'':
\begin{equation}\label{eq:bicomplex}
  \xymatrix{
    0 & \ar[l] \bigoplus_{\cL_1\boxtimes\cL_2 \in \cR_0} \cech_{\F\otimes\cL_1} \otimes \cL_2 & \ar[l] \cdots
    &   \ar[l] \bigoplus_{\cL_1\boxtimes\cL_2 \in \cR_N} \cech_{\F\otimes\cL_1} \otimes \cL_2 & \ar[l] 0}.
\end{equation}
Since the differentials of $\cech_\G$ have entries in $\kk$, the columns of \eqref{eq:bicomplex} split. Thus, we may apply \cite[Lem.~3.5]{EFS03} to conclude that the totalization of \eqref{eq:bicomplex} is homotopy equivalent to a complex $\B(\F)$ concentrated in degrees $k=-N,\dots,N$ with terms
\begin{equation}\label{eq:BM}
  \B(\F)_k
  = \bigoplus_{i-j=k} \bigoplus_{\cL_1\boxtimes\cL_2 \in \cR_i} H^j(X, \F\otimes\cL_1) \otimes\cL_2
  \cong \bigoplus_{i-j=k} \RR^j{\pi_2}_*(\pi_1^*\F\otimes\cR_i).
\end{equation}
The terms of $\B(\F)$ arise from the totalization of the vertical homology of \eqref{eq:bicomplex}. 

Over projective space, the analogue of this Fourier--Mukai transform involving Beilinson's resolution of the diagonal is called the Beilinson monad (see e.g.~\cite{EFS03}), hence the notation $\B(-)$. Note that ``the'' complex $\B(\F)$ is only well-defined up to homotopy equivalence, since the differential depends on a choice of splitting of the columns in the bicomplex \eqref{eq:bicomplex}. More precisely, for each term $Y_{i,j}$ of \eqref{eq:bicomplex}, choose a decomposition $Y_{i,j} = B_{i,j} \oplus H_{i,j} \oplus L_{i,j}$ such that $B_{i,j} \oplus H_{i,j} = Z^{\on{vert}}_{i,j}$, where $Z^{\on{vert}}_{i,j}$ denotes the vertical cycles in $Y_{i,j}$. Notice that there is a canonical isomorphism $H_{i,j} \cong \bigoplus_{\cL_1 \boxtimes \cL_2\in \cR_i}  H^{-j}(\F \otimes \cL_1) \otimes \cL_2 $. Let $\sigma_H\colon Y_{\bullet, \bullet} \to H_{\bullet, \bullet}$ and $\sigma_B\colon Y_{\bullet, \bullet} \to B_{\bullet, \bullet}$ denote the projections, let $g\colon L_{\bullet, \bullet} \xto{\cong} B_{\bullet ,\bullet -1}$ denote the isomorphism induced by the vertical differential, and let $\pi = g^{-1} \sigma_B$. By \cite[Lem.~3.5]{EFS03}, the differential on $\B(\F)$ is given by
\[ \del_{\B(\F)} = \sum_{i \ge 0} \sigma_H(d_{\on{hor}} \pi)^i d_{\on{hor}}, \]
where $d_{\on{hor}}$ is the horizontal differential in the bicomplex \eqref{eq:bicomplex}.

\begin{remark}
  The $i = 0$ term in the formula for $\del_{\B(\F)}$ is simply the map induced by the differential on $\cR$; it is independent of the choices of splittings of the columns of \eqref{eq:bicomplex}. Since this is the only part of the differential on $\B(\F)$ that we will need to explicitly compute, we will ignore the ambiguity of $\B(\F)$ up to homotopy equivalence from now on.
\end{remark}

\subsection{Proof of Theorem~\ref{thm:main}}\label{sec:proof-subsection}

\begin{proof}
  To prove parts (1) and (2), first recall that $R$ is the direct sum of the degree $(d_1,d_2,d_3,d_4,0,0)$ components of $K(0,0,0,0,r,s)$ for all $d_1,\dots,d_4\in\ZZ$. Thus, since $K$ is exact in positive homological degrees, $R$ is as well; moreover, the differentials of $R$ are linear\footnote{Free complexes that are linear in the sense of Theorem~\ref{thm:main}(2) are called \emph{strongly linear} in \cite{BE22}.}. We now check that $R$ has property (3). For all $k, \ell \in \ZZ$, we have
\begin{equation}\label{eq:H0}
  \dim_\kk H^0(\FF_{a_s}, \OO(k, \ell)) = \begin{cases}
    (k + 1)(\ell + 1) + \binom{\ell + 1}{2}a_s, & \ell \ge 0; \\
    0,                                          & \ell  <  0.
  \end{cases}
\end{equation}
We now compute:
\begin{align*}
\on{rank} \cR_n &= \sum_{ i = 0}^n \binom{r+1}{i} \binom{s + 1}{n - i} \dim_\kk H^0(\FF_{a_s}, \OO(r - i, s- (n - i))) \\
&= \sum_{i = 0}^r \binom{r+1}{i} \binom{s + 1}{n - i} \left( (r -i+1)(s - (n - i) + 1) + \binom{s - (n - i) +1}{2}a_s \right)  \\
&+ \binom{s+1}{n - (r+1)} \binom{s - (n - (r+1)) + 1}{2}a_s \\
&=  \sum_{ i = 0}^r \binom{r}{i} \binom{s}{n - i}(r+1)(s+1) +  \sum_{ i = 0}^r \binom{r+1}{i} \binom{s - 1}{n - i} \binom{s + 1 }{2}a_s \\
&+ \binom{s-1}{n - (r+1)} \binom{s + 1}{2}a_s\\
&=  \sum_{ i = 0}^r \binom{r}{i} \binom{s}{n - i}(r+1)(s+1) + \sum_{ i = 0}^{r+1} \binom{r+1}{i} \binom{s - 1}{n - i} \binom{s + 1 }{2}a_s \\
&= \binom{r + s}{n} \dim_\kk H^0(\FF_{a_s}, \OO(r, s)).
\end{align*}
The first equality follows from the definition of $\cR$, the second from \eqref{eq:H0}, the third from some straightforward manipulations, the fourth by combining the second and third terms, and the last by Vandermonde's identity and the equality $\dim_\kk H^0(\FF_{a_s}, \OO(r, s)) = (r+1)(s+1) + \binom{s+1}{2}a_s$. This proves (3).

Finally, we check property (4): namely, that the cokernel of the differential $\del_1\colon\cR_1\to\cR_0$ is $\OO_\Delta$. Just as in the proof of \cite[Prop.~3.2]{CK08}, we will prove that $\cR$ is a resolution of $\OO_\Delta$ by showing there is a chain map $\cR\to\OO_\Delta$ that induces a natural isomorphism on certain Fourier--Mukai transforms. In detail: given any $i,j\in\ZZ$, there is a natural map $\OO(i,j,-i,-j)\to\OO_\Delta$ given by multiplication. These maps determine a natural map $\nu_0\colon\cR_0\to\OO_\Delta$, and it is clear from the description of $\del_1$ in Example~\ref{ex:firstdiff} that $\nu_0$ determines a chain map $\nu\colon\cR\to\OO_\Delta$. Recall that $\Phi_{\cR}$ denotes the Fourier--Mukai transform associated to $\cR$. To show that $\nu$ is a quasi-isomorphism, we need only prove that the induced natural transformation $\Phi_\nu \colon\Phi_\cR \to \Phi_{\OO_\Delta}$ on Fourier--Mukai transforms is a natural isomorphism; indeed, this immediately implies that $\Phi_{\on{cone}(\nu)}$ is isomorphic to the 0 functor, and so $\on{cone}(\nu) = 0$ by \cite[Lem.~2.1]{CK08}.

The category $\Db(X)$ is generated by the line bundles $\OO(b,c)$ with $0 \le b \le r$ and $0 \le c \le s$; in fact, these bundles form a full exceptional collection in $\Db(X)$ \cite[Cor.~2.7]{Orlov93}. Since $\Phi_{\OO_\Delta}$ is the identity functor, we need only show that the map $\Phi_\cR(\OO(b, c)) \to \OO(b,c)$ induced by $\Phi_\nu$ is an isomorphism in $\Db(X)$.

Say $\OO(d_1, d_2, d_3, d_4)$ is a summand of $\cR$. We first show that the line bundle $\OO(d_1 + b, d_2 + c)$ on $X$ is acyclic, i.e. $H^i(X,\OO(d_1 + b, d_2 + c)) = 0$ for $i > 0$. Say the summand $\OO(d_1, d_2, d_3, d_4)$ of $\cR$ corresponds to the monomial $\a_{i_1} \cdots \a_{i_k} \b_{j_1} \cdots \b_{j_\ell} m$, where $k \le r + 1$, $\ell \le s + 1$, and $m \in M_{r-k, s - \ell}$. It follows that $d_1 = -k - t_1$ and $d_2 = -\ell - t_2$ for some $t_1 \le r - k$ and $t_2 \le s - \ell$. In particular, we have $d_1 + b \ge d_1 \ge -r$, and $d_2 + c \ge d_2 \ge -s$. Thus, $\OO(d_1 + b, d_2 + c)$ satisfies either (a) or (b) in Corollary~\ref{cor:vanishing}(2), and so $\OO(d_1 + b, d_2 + c)$ is acyclic.

Recall from \S\ref{sec:FM} that, given any sheaf $\F$ on $X$, $\Phi_\cR(\F)$ may be modeled explicitly as the complex $\B(\F)$. The previous paragraph implies that the terms in $\B(\OO(b,c))$ involving higher cohomology vanish; that is, the nonzero terms of $\B(\OO(b,c))$ are of the form $H^0(\cL_1(b,c)) \otimes \cL_2$, where $\cL_1 \boxtimes \cL_2$ is a summand of $\cR$. In particular, $\B(\OO(b,c))$ is concentrated in nonnegative degrees, the map $\B_0(\OO(b, c)) \to \OO(b,c)$ induced by $\nu$ is the natural multiplication map, and the differential on $\B(\OO(b,c))$ is induced by the differential on $\cR$. It follows that $\B(\OO(b,c))$ is exact in positive degrees, since $\cR$ has this property. We now show, by direct computation, that the induced map $H_0(\B(\OO(b,c))) \to \OO(b,c)$ is an isomorphism.

It follows from our explicit descriptions of the terms $R_0$ and $R_1$ in Example~\ref{ex:firstdiff} that
\begin{align*}
  \B(\OO(b,c))_0 &= \bigoplus_{m \in M_{r,s}} H^0(X, \OO(b - d_1^m, c - d_2^m)) \otimes \OO(d_1^m, d_2^m) \cdot m, \quad \text{and} \\
  \B(\OO(b,c))_1 &= \B(\OO(b,c))_1^\a \oplus \B(\OO(b,c))_1^\b, \quad \text{where} \\
  \B(\OO(b,c))^\a_1 &= \bigoplus_{i = 0}^r \bigoplus_{ m \in M_{r -1, s}} H^0(X, \OO(b -d_1^m -1,c-d_2^m)) \otimes \OO(d_1^m, d_2^m)\cdot \a_i m, \\
  \B(\OO(b,c))^\b_1 &= \bigoplus_{i = 0}^s \bigoplus_{ m \in M_{r , s-1}} H^0(X, \OO(b -d_1^m +a_i,c-d_2^m-1)) \otimes \OO(d_1^m, d_2^m)\cdot \b_i m.
\end{align*}
We represent the first differential on $\B(\OO(b,c))$ as a matrix with respect to the above decomposition, along with the monomial bases of each cohomology group. The column of this matrix corresponding to $\a_im$ and a monomial $z$ in the Cox ring $S = k[x_0, \dots, x_r, y_0, \dots,y_s]$ of $X$ of degree $(b - d_1^m - 1, c-d_2^m)$ has exactly two nonzero entries:
  \begin{itemize}
  \item an entry of $1$ for $u_2m \in M_{r, s}$ and $x_i z \in H^0(X, \OO(b-d_1^{u_2m}, c - d_2^{u_2m}))$;
  \item an entry of $-x_i'$ for $u_0m \in M_{r, s}$ and $z \in  H^0(X, \OO(b-d_1^{u_0m}, c - d_2^{u_0m}))$.
  \end{itemize}
  Similarly, the column corresponding to $\b_im$ and a monomial $w \in S$ of degree $(b-d_1^m + a_i, c-d_2^m - 1)$ has exactly two nonzero entries:
  \begin{itemize}
  \item an entry of $1$ for $u_3m$ and $y_iw\in H^0(X, \OO(b-d_1^{u_3m}, c - d_2^{u_3m}))$;
  \item an entry of $-y_i'$ for $u_0^{a_s - a_i}u_1u_2^{a_i}m$ and $w \in H^0(X, \OO(b-d_1^{u_0^{a_s - a_i}u_1u_2^{a_i}m}, c - d_2^{u_0^{a_s - a_i}u_1u_2^{a_i}m}))$.
  \end{itemize}
  That is, the first differential on $\B(\OO(b, c))$ has the following form:
  \begin{footnotesize} \[
  \begin{pNiceMatrix}[last-col,last-row,nullify-dots]
         &  0    &        &  0    &        & \Vdots \\
         & -x_i' &        &  0    &        & \;\;\;\;\;z \otimes u_0m \\
         &  0    &        &  0    &        & \Vdots \\
         &  0    &        & -y_i' &        & \;\;\;\;w \otimes u_0^{a_s - a_i}u_1u_2^{a_i}m \\
  \cdots\hspace*{-8mm} & 0 & \hspace*{-8mm}\cdots\hspace*{-8mm} & 0 & \hspace*{-8mm}\cdots & \Vdots \\
         &  1    &        &  0    &        & \;\,x_iz \otimes u_2m \\
         &  0    &        &  0    &        & \Vdots \\
         &  0    &        &  1    &        & \;y_iw \otimes u_3m \\
         &  0    &        &  0    &        & \Vdots \\
         &  z \otimes \a_i m & \hspace*{-8mm}\cdots\hspace*{-8mm} & w \otimes \b_i m &  &
  \end{pNiceMatrix}.
  \] \end{footnotesize}
  
\noindent Now observe: every column of this matrix contains exactly one ``1", and there is exactly one row that does not contain a ``1": namely, the row corresponding to the summand
$H^0(X, \OO) \otimes \OO(b, c) \cdot u_0^{b+ca_s}u_1^cu_2^{r-b}u_3^{s-c}.$
It follows immediately that the cokernel of this matrix is isomorphic to the summand $H^0(X, \OO) \otimes \OO(b, c)$, and the multiplication map induced by $\nu$ from this summand to $\OO(b, c)$ is clearly an isomorphism.
\end{proof}

\begin{remark}\label{rem:bigres}
  Our construction of the resolution $\cR$ realizes it as a subcomplex of the (infinite rank) resolution of the diagonal obtained in \cite[Thm.~4.1]{BE21} and therefore yields a positive answer to \cite[Conj.~7.2]{BE21} for smooth projective toric varieties of Picard rank 2.
\end{remark}

\begin{corollary}\label{cor:monad}
  Given a coherent sheaf $\F$ on $X$, we have $\B(\F) \cong \F$ in $\Db(X)$.
\end{corollary}

\begin{corollary}\label{cor:cK}
  Consider the ideal $I = (\a_0, \dots, \a_r, \b_0, \dots, \b_s) \subseteq S_E$, and let $\mathcal{D}$ denote the sheaf $\widetilde{S_E/I}$ on $E$. We have an isomorphism $\pi_*\mathcal{D}(0,0,0,0,r,s)\cong \OO_\Delta$ of sheaves on $X \times X$.
\end{corollary}
\begin{proof}
Recall that $\cK$ is the sheafification of the Koszul complex on the generators of $I$, which form a regular sequence. Therefore $\cK$ is a locally free resolution of $\mathcal{D}$, and using Proposition~\ref{prop:higherdirect} and Theorem~\ref{thm:main}(4) we have $\pi_*\mathcal{D}(0,0,0,0,r,s)\cong\pi_*\cK(0,0,0,0,r,s)\cong\cR\cong\OO_\Delta$.
\end{proof}

We will now prove Conjecture~\ref{conj:BES} for $X$ as in Theorem~\ref{thm:main}.

\begin{proof}[Proof of Corollary~\ref{cor:rank2}]
  Our proof is nearly the same as that of \cite[Prop.~1.2]{BES20}. Given a finitely generated graded module $M$ over the Cox ring of $X$, let $\F$ be the associated sheaf on $X$. Applying the Fujita Vanishing Theorem, choose $i,j\gg0$ such that, for all summands $\cL_1\boxtimes\cL_2$ of the resolution of the diagonal $\cR$ from Theorem~\ref{thm:main}, we have $H^q(X, \F(i,j)\otimes\cL_1) = 0$ for $q>0$. The complex $\B(\F(i,j))$ is a resolution of $\F(i,j)$ of length at most $\dim(X)$ consisting of finite sums of line bundles, and twisting back by $(-i, -j)$ gives a resolution of $\F$. Now applying the functor $\G \mapsto \bigoplus_{(k, \ell) \in \ZZ^2} H^0(X, \G(k, \ell))$ to the complex $\B(\F(i, j))(-i,-j)$ gives a virtual resolution of $M$.
\end{proof}

\section{A Horrocks-type splitting criterion}\label{sec:horrocks}

Let $X$ denote the projective bundle $\PP(\OO \oplus \OO(a_1) \oplus \cdots \oplus \OO(a_s))$ over $\PP^r$, where $a_1 \le \cdots \le a_s$. Given a coherent sheaf $\F$ on $X$, let $\B(\F)$ be the complex of sheaves on $X$ defined in \S\ref{sec:FM}. Recall from the introduction the notation $\g(\F)$ for the cohomology table of $\F$. We will need the following technical result.

\begin{lemma}[{cf.~\cite[Lem.~7.3]{EES15}}]\label{lem:induction}
  Let $\E$ be a vector bundle on $X$, and suppose we have $\g(\E) = \g(\OO^m) + \g(\E')$ for some vector bundle $\E'$ on $X$ with $\B(\E')_1 = 0$. There is an isomorphism $\E \cong \OO^m \oplus \E''$ for some vector bundle $\E''$ such that $\g(\E'') = \g(\E')$.
\end{lemma}
\begin{proof}
  Let $\cR$ be the resolution of the diagonal for $X$ constructed in \S\ref{sec:bundle}. We have a Beilinson-type spectral sequence
  \[ E_1^{-i,j}(\E) = \RR^j{\pi_2}_*(\pi_1^*\E\otimes\cR_i) \Rightarrow \RR^{i-j}{\pi_2}_*(\pi_1^*\E\otimes\cR)
  \cong \begin{cases}
    \E, & i = j; \\
    0, & i \ne j. 
  \end{cases} \]
  The first page looks as follows:
  \begin{equation}\label{eq:derived-E1}
    \begin{tikzcd}[column sep=small, row sep=tiny]
	{} & \vdots & \vdots & \vdots & {} \\[-5pt]
	{} & {\RR^2{\pi_2}_*(\pi_1^*\E\otimes\cR_0)} & {\RR^2{\pi_2}_*(\pi_1^*\E\otimes\cR_1)} &     {\RR^2{\pi_2}_*(\pi_1^*\E\otimes\cR_2)} & \cdots \\
	{} & {\RR^1{\pi_2}_*(\pi_1^*\E\otimes\cR_0)} & {\RR^1{\pi_2}_*(\pi_1^*\E\otimes\cR_1)} & \red{\RR^1{\pi_2}_*(\pi_1^*\E\otimes\cR_2)} & \cdots \\
	{} & {   {\pi_2}_*(\pi_1^*\E\otimes\cR_0)} & \red{ {\pi_2}_*(\pi_1^*\E\otimes\cR_1)} &     {     {\pi_2}_*(\pi_1^*\E\otimes\cR_2)} & \cdots \\
	{} & {} & {}
	\arrow[from=4-4, to=4-3]
	\arrow[from=4-3, to=4-2]
	\arrow[from=4-5, to=4-4]
	\arrow[from=3-4, to=3-3]
	\arrow[from=3-5, to=3-4]
	\arrow[from=2-5, to=2-4]
	\arrow[from=2-4, to=2-3]
	\arrow[from=2-3, to=2-2]
	\arrow[from=3-3, to=3-2]
  \arrow[crossing over, shift left=17, shorten <=-28pt, shorten >=-18pt, dashed, no head, from=4-2, to=1-2]
	\arrow[shift right=5, shorten <=-87pt, shorten >=-20pt, dashed, no head, from=4-2, to=4-5]
	\arrow["{k=2}"{description, sloped, pos=0.70}, shift left=4, shorten >=70pt, dotted, no head, from=3-5, to=5-3]
	\arrow["{k=1}"{description, sloped, pos=0.81}, shift left=4, shorten >=70pt, dotted, no head, from=2-5, to=5-2]
	\arrow["{k=0}"{description, sloped, pos=0.95}, shift left=1, shorten >=30pt, dotted, no head, from=1-5, to=5-1]
	\arrow[shift right=2, shorten <=25pt, shorten >=23pt, dotted, no head, from=1-3, to=3-1]
	\arrow[shift right=2, shorten <=25pt, shorten >=23pt, dotted, no head, from=1-4, to=4-1]
    \end{tikzcd}
  \end{equation}
  Notice that $E_1^{-i,j}(\E) \cong \bigoplus H^j(X, \E\otimes\cL_1)\otimes\cL_2$, where the direct sum ranges over the summands $\cL_1\boxtimes\cL_2$ of $\cR_i$. It follows that $\B(\E)_1 = \bigoplus_{i-j=1} E_1^{-i,j}(\E)$. Moreover, since the terms of the first page only depend on $\gamma(\E)$, we have
  \[ E_1^{-i,j}(\E) = E_1^{-i,j}(\OO)^m \oplus E_1^{-i,j}(\E'). \]
Observe that $E_1^{0,0}(\OO) = \OO$, and $E_1^{-i,j}(\OO) = 0$ when either $i\ne0$ or $j\ne0$. In particular, we have $E_1^{0,0}(\E) = \OO^m \oplus E_1^{0,0}(\E')$, and it follows from the hypothesis $\B(\E')_1=0$ that the terms along the $k=1$ diagonal in \eqref{eq:derived-E1} (colored in red) vanish. Thus, every differential in the spectral sequence with either source or target given by $E_r^{0,0}(\E)$ for some $r$ vanishes. We conclude that $\OO^m$ is a summand of $E_{\infty}^{0,0}(\E)$, and hence $\E$ as well. 
\end{proof}

\begin{remark}
 A similar technique was recently utilized by Bruce, Cranton Heller, and the second author in \cite{BCHS21} to give a characterization of multigraded Castelnuovo--Mumford regularity on products of projective spaces. An interesting question is whether there is a similar result for smooth projective varieties of Picard rank 2 using the resolution $\cR$.
\end{remark}

We will now prove our splitting criterion.

\begin{proof}[Proof of Theorem~\ref{thm:splitting}]
By induction, it suffices to show that $\OO(b_1, c_1)^{m_1}$ is a summand of $\E$. Without loss of generality, we may assume $(b_2, c_2) < (b_1, c_1)$. We may also twist $\E$ so that $b_1=c_1=0$, which implies $(b_i, c_i) < 0$ for all $i > 1$. Suppose $(a, b) < 0$. By Lemma~\ref{lem:induction}, it suffices to show that $\B(\OO(a, b))_1 = 0$. This amounts to showing that, for $0<n\le r+s$, we have
\begin{equation}\label{eq:condition}
  H^{n-1}(X, \OO(a, b)\otimes\cL_1)=0 \text{ when } \cL_1 \boxtimes \cL_2 \text{ is a summand of }\cR_{n}.
\end{equation}
By Corollary~\ref{cor:vanishing}(1)(a), we need only show that \eqref{eq:condition} holds for $n \in \{1,r+1,s+1\}$. We recall that any summand of $\cR_n$ corresponds to a monomial of the form
\begin{equation}\label{eq:summand}
  \a_{i_1} \cdots \a_{i_e}\cdot \b_{j_1} \cdots \b_{j_f} \cdot m,
\end{equation}
where $e + f = n$, and $m \in M_{r - e, s - f}$ (using Notation~\ref{not:monomials}). Writing $m = u_0^{c_0} u_1^{c_1} u_2^{c_2} u_3^{c_3}$, we have that the summand of $\cR_n$ corresponding to \eqref{eq:summand} is $\OO(-e - d_1, -f - d_2, d_1, d_2)$, where $d_1 = c_0 -a_sc_1$ and $d_2 = c_1$. In particular, we have $0 \le d_2 \le s-f$, which immediately implies that $-s \le -f - d_2 \le 0$. When $-f -d_2 < 0$, \eqref{eq:condition} holds for $n \in \{1, r+1,s+1\}$ by Corollary~\ref{cor:vanishing}(1)(b - d). Suppose $-f - d_2 = 0$. Since $f, d_2 \ge 0$, we have $f = 0 = d_2$. It follows that $e = n$ and $d_1 = c_0 \ge 0$. Corollary~\ref{cor:vanishing}(1)(b) therefore implies that \eqref{eq:condition} holds when $n = 1$. Corollary~\ref{cor:vanishing}(1)(c) implies that \eqref{eq:condition} holds for $n = s+1$ when $r \ne s$; we may thus reduce to the case where $n = r+1$. But this case cannot occur, since there is no $m \in M_{-1, s}$ of the form $u_0^{c_0} u_2^{c_2} u_3^{c_3}$.
\end{proof}

\begin{remark}
  If we replace the nef ordering with the effective ordering in the statement of Theorem~\ref{thm:splitting}, our proof fails. The problem arises in the final step: there exist line bundles $\OO(b, c) \in \Pic{X}$ such that $-(b, c)$ is effective but $\B(\OO(b, c))_1 \ne 0$. For instance, over a Hirzebruch surface of type $a$, the divisor $-(a, -1)$ is effective, and $\B(\OO(a, -1))_1 \ne 0$.
\end{remark}


\let\o\slasho 

\bibliography{references}

\end{document}